\documentclass[12pt]{article}
\usepackage{amsfonts}
\usepackage{amsmath}
\usepackage{color}

\newtheorem{theorem}{Theorem}

\newtheorem{example}[theorem]{Example}

\newtheorem{lemma}[theorem]{Lemma}

\newenvironment{proof}[1][Proof]{\noindent\textbf{#1.} }{\ \rule{0.5em}{0.5em}}
\input{tcilatex}
\begin{document}

\title{Positive solutions for anisotropic discrete boundary value problems}
\author{Marek Galewski, Szymon G\l \c{a}b, Renata Wieteska}
\maketitle

\begin{abstract}
Using mountain pass arguments and the Karsuh-Kuhn-Tucker Theorem, we prove
the existence of at least two positive solution of the anisotropic discrete
Dirichlet boundary value problem. Our results generalize and improve those
of \cite{TG}.
\end{abstract}

\textbf{Math Subject Classifications}: 39A10, 34B18, 58E30.


\section{Introduction}

In this note we consider an anisotropic difference equation with Dirichlet
type boundary condition on the form 
\begin{equation}
\left\{ 
\begin{array}{l}
\Delta \left( |\Delta y(k-1)|^{p(k-1)-2}\Delta y(k-1)\right) +f(k,y(k))=0,%
\text{ }k\in \left[ 1,T\right] \bigskip , \\ 
y(0)=y(T+1)=0,%
\end{array}%
\right.  \label{zad}
\end{equation}%
where $T\geq 2$ is a integer, $f:[1,T]\times 
%TCIMACRO{\U{211d} }%
%BeginExpansion
\mathbb{R}
%EndExpansion
\rightarrow (0,+\infty )$ is a continuous function; $[1,T]$ is a discrete
interval $\{1,2,...,T\},$ $\Delta y(k-1)=y\left( k\right) -y(k-1)$ is the
forward difference operator; $y\left( k\right) \in 
%TCIMACRO{\U{211d} }%
%BeginExpansion
\mathbb{R}
%EndExpansion
$ for all $k\in \left[ 1,T\right] $; $p:\left[ 0,T+1\right] \rightarrow
\lbrack 2,+\infty )$. Let $p^{-}=\min_{k\in \left[ 0,T+1\right] }p\left(
k\right) $; $p^{+}=\max_{k\in \left[ 0,T+1\right] }p\left( k\right) $.$%
\bigskip $

About the nonlinear term, we assume the following condition$\bigskip $

\textbf{(C.1)} There exist a number $m>p^{+}$ and \ functions $\varphi
_{1},\varphi _{2}:[1,T]\rightarrow (0,\infty )$\textit{, }$\psi _{1},\psi
_{2}:[1,T]\rightarrow \left( 0,\infty \right) $ such that%
\begin{equation*}
\psi _{1}(k)+\varphi _{1}(k)|y|^{m-2}y\leq f(k,y)\leq \varphi
_{2}(k)|y|^{m-2}y+\psi _{2}(k)
\end{equation*}%
for all $y\geq 0$ and all $k\in \left[ 1,T\right] .\bigskip $

Now, we will show the example of a function which satisfies condition 
\textbf{(C.1)}.

\begin{example}
\label{przyklad}Let $f:[1,T]\times 
%TCIMACRO{\U{211d} }%
%BeginExpansion
\mathbb{R}
%EndExpansion
\rightarrow (0,\infty )$ be given by 
\begin{equation*}
\bigskip f(k,y)=\left\vert y\right\vert ^{m-2}y\frac{2+arctg(y)}{T^{2}k}+%
\frac{\sin ^{2}(k)e^{-\left\vert y\right\vert }+1}{T^{3}}
\end{equation*}%
\textit{for }$\left( k,y\right) \in \lbrack 1,T]\times 
%TCIMACRO{\U{211d} }%
%BeginExpansion
\mathbb{R}
%EndExpansion
$; here $m>p^{+}$. We see that\texttt{\ }for $y\geq 0$\ we have%
\begin{equation*}
\frac{1}{T^{3}}+\frac{2}{T^{2}k}\left\vert y\right\vert ^{m-2}y\leq
f(k,y)\leq \frac{4+\pi }{2T^{2}k}\left\vert y\right\vert ^{m-2}y+\frac{2}{%
T^{3}}.
\end{equation*}%
Thus we may put 
\begin{equation*}
\varphi _{1}(k)=\frac{2}{T^{2}k};\text{ }\varphi _{2}(k)=\frac{4+\pi }{%
2T^{2}k};\text{ }\psi _{1}(k)=\frac{1}{T^{3}};\text{ }\psi _{2}(k)=\frac{2}{%
T^{3}}.
\end{equation*}
\end{example}

Solutions to \eqref{zad}\ will be investigated in a space

\begin{equation*}
Y=\{y:[0,T+1]\rightarrow \mathbb{R}:y(0)=y(T+1)=0\}
\end{equation*}%
considered with a norm 
\begin{equation*}
\Vert y\Vert =\left( \sum_{k=1}^{T+1}|\Delta y(k-1)|^{2}\right) ^{1/2}
\end{equation*}%
with which $Y$ becomes a Hilbert space. For $y\in Y$ \ let 
\begin{equation*}
y_{+}=\max \{y,0\},\text{ \ \ }y_{-}=\max \{-y,0\}.
\end{equation*}%
Note that 
\begin{equation*}
y_{+}\geq 0\text{ and }y_{-}\geq 0;\text{ \ \ }y=y_{+}-y_{-};\text{ \ \ }%
y_{+}\cdot y_{-}=0.
\end{equation*}

In order to demonstrate that problem \eqref{zad} has at least two positive
solutions we assume additionally the following condition

\textbf{(C.2)}\ $\ \ \ \ \ \ \ T^{\frac{p^{+}-2}{2}}\left( \frac{1}{\sqrt{T+1%
}}\right) ^{p^{+}}>\sum\limits_{k=1}^{T}\left( \varphi _{2}(k)+\psi
_{2}(k)\right) .\bigskip $

\begin{example}
We show how assumption \textbf{(C.2)} is verified in Example \ref{przyklad}.
Taking $p^{+}=18$ and $T=200$ we see that 
\begin{equation*}
T^{\frac{p^{+}-2}{2}}\left( \frac{1}{\sqrt{T+1}}\right)
^{p^{+}}=0.009>0.002=\sum\limits_{k=1}^{T}\left( \varphi _{2}(k)+\psi
_{2}(k)\right) .
\end{equation*}
\end{example}

\begin{theorem}
\label{maintheorem}Suppose that assumptions \textbf{(C.1) }and\textbf{\ (C.2)%
} hold.{\ }Then \eqref{zad} has at least two positive solutions.
\end{theorem}

Discrete BVPs received some attention lately. Let us mention, far from being
exhaustive, the following recent papers on discrete BVPs investigated via
critical point theory, \cite{agrawal}, \cite{CIT}, \cite{caiYu}, \cite{Liu}, 
\cite{sehlik}, \cite{TianZeng}, \cite{teraz}, \cite{zhangcheng}, \cite%
{nonzero}. The tools employed cover the Morse theory, mountain pass
methodology, linking arguments, i.e. methods usually applied in continuous
problems.\bigskip

Continuous versions of problems like (\ref{zad}) are known to be
mathematical models of various phenomena arising in the study of elastic
mechanics (see \cite{B}), electrorheological fluids (see \cite{A}) or image
restoration (see \cite{C}). Variational continuous anisotropic problems have
been started by Fan and Zhang in \cite{D} and later considered by many
methods and authors (see \cite{hasto} for an extensive survey of such
boundary value problems). The research concerning the discrete anisotropic
problems of type (\ref{zad}) have only been started (see \cite{KoneOuro}, 
\cite{MRT} where known tools from the critical point theory are applied in
order to get the existence of solutions).\bigskip

When compared with \cite{TG} we see that our problem is more general since
we consider variable exponent case instead of a constant one. While we do
not include term depending on $\Phi _{p^{-}}(y)=|y|^{p^{-}-2}y$ in the
nonlinear part as is the case in \cite{TG}, it is apparent that our results
would also hold should we have made our nonlinearity more complicated. We
note that term $\Phi _{p^{-}}(y)=|y|^{p^{-}-2}y$ does not influence the
growth of the nonlinearity.\bigskip

\section{Auxiliary results}

We connect positive solutions to (\ref{zad}) with critical points of
suitably chosen action functional. Let 
\begin{equation*}
F(k,y)=\int_{0}^{y}f(k,s)ds\text{ for }y\in \mathbb{R}\text{ and }k\in \left[
1,T\right] \text{.}
\end{equation*}%
Let us define a functional $J:Y\rightarrow R$\ by the formula%
\begin{equation*}
J(y)=\sum_{k=1}^{T+1}\frac{1}{p(k-1)}|\Delta
y(k-1)|^{p(k-1)}-\sum_{k=1}^{T}F(k,y_{+}(k)).
\end{equation*}%
Functional $J$\ is sligthly different from functionals applied in
investigating the existence of positive solutions, compare with \cite%
{TianZeng}. Thus we indicate its properties. The functional $J$\ is
continuously G\^{a}teaux differentiable and its G\^{a}teaux derivative $%
J^{\prime }$\ at $y$\ reads\texttt{\ }%
\begin{equation}
\begin{array}{l}
\langle J^{\prime }(y),v\rangle =\sum\limits_{k=1}^{T+1}|\Delta
y(k-1)|^{p(k-1)-2}\Delta y(k-1)\Delta v(k-1)-\bigskip \\ 
\sum\limits_{k=1}^{T}f(k,y_{+}(k))v(k)%
\end{array}
\label{functional1}
\end{equation}%
for all $v\in Y$. Suppose that $y$\ is a critical point to $J$, i.e. $%
\langle J^{\prime }(y),v\rangle =0$\ for all $v\in Y$. Summing by parts and
taking boundary values into account, see \cite{GW}, we observe that 
\begin{equation*}
\begin{array}{l}
0=-\sum\limits_{k=1}^{T+1}\Delta (|\Delta y(k-1)|^{p(k-1)-2}\Delta
y(k-1))v(k)-\bigskip \\ 
\sum\limits_{k=1}^{T}f(k,y_{+}(k))v(k).%
\end{array}%
\end{equation*}%
Since $v\in Y$\ is arbitrary\ we see that $y$\ satisfies (\ref{zad}).\bigskip

Now, we recall some auxiliary materials which we use later on, \textbf{(A.1)}%
-\textbf{(A.3)} see \cite{MRT}, \textbf{(A.4)}, \textbf{(A.5)} see \cite{GW}%
, \textbf{(A.6)} see \cite{TianZeng}:\bigskip

\textbf{(A.1)} \textit{For every} $y\in Y$ \textit{with} $\Vert y\Vert >1$ 
\textit{we have }%
\begin{equation*}
\sum_{k=1}^{T+1}|\Delta y(k-1)|^{p(k-1)}\geq T^{\frac{2-p^{-}}{2}}\Vert
y\Vert ^{p^{-}}-T.\bigskip
\end{equation*}

\textbf{(A.2)} \textit{For every }$\mathit{\ }y\in Y$ \textit{with} $\Vert
y\Vert \leq 1$ \textit{we have }%
\begin{equation*}
\sum\limits_{k=1}^{T+1}\left\vert \Delta y(k-1)\right\vert ^{p(k-1)}\geq T^{%
\frac{p^{+}-2}{2}}\Vert y\Vert ^{p^{+}}\text{.}\bigskip
\end{equation*}

\textbf{(A.3) }\textit{For every }$y\in Y$ \textit{and} \textit{for any}$%
\mathit{\ }m\geq 2$ \textit{we have }

\begin{equation*}
(T+1)^{\frac{2-m}{2}}\left\Vert y\right\Vert ^{m}\leq
\sum\limits_{k=1}^{T+1}\left\vert \Delta y(k-1)\right\vert ^{m}\leq
(T+1)\left\Vert y\right\Vert ^{m}.\bigskip
\end{equation*}

\textbf{(A.4)}\textit{\ If }$p^{+}\geq 2$\textit{,\ there exists }$%
C_{p^{+}}>0$ \textit{\ such that for every } $y\in Y$

\textit{\ }%
\begin{equation*}
\sum\limits_{k=1}^{T+1}\left\vert \Delta y(k-1)\right\vert ^{p(k-1)}\leq
2^{p^{+}}(T+1)\left( C_{p^{+}}\left\Vert y\right\Vert ^{p^{+}}+1\right) 
\text{.\ }\bigskip
\end{equation*}

\textbf{(A.5)} \textit{For every }$y\in Y$ \textit{and} \textit{for any }$%
m\geq 2$\textit{\ } \textit{we have }

\begin{equation*}
\sum\limits_{k=1}^{T+1}\left\vert \Delta y(k-1)\right\vert ^{m}\leq
2^{m}\sum\limits_{k=1}^{T}\left\vert y(k)\right\vert ^{m}.\bigskip
\end{equation*}

\textbf{(A.6)} \textit{For every }$y\in Y$ \textit{and} \textit{for any }$%
p,q>1$\textit{\ such that }$\frac{1}{p}+\frac{1}{q}=1$ \textit{we have }%
\begin{equation*}
\left\Vert y\right\Vert _{C}=\max_{k\in \lbrack 1,T]}\left\vert
y(k)\right\vert \leq (T+1)^{\frac{1}{q}}\left( \sum_{k=1}^{T+1}|\Delta
y(k-1)|^{p}\right) ^{1/p}.\bigskip \text{ }
\end{equation*}

Let $E$ be a real Banach space. We say that a functional $J:E\rightarrow 
\mathbb{R}$ satisfies Palais-Smale condition if every sequence $(y_{n})$
such that $\{J(y_{n})\}$ is bounded and $J^{\prime }(y_{n})\rightarrow 0$,
has a convergent subsequence.\bigskip

\begin{lemma}
\cite{mp}\label{lem2} Let $E$ be a Banach space and $J\in C^{1}(E,\mathbb{R}%
) $ satisfy Palais-Smale condition. Assume that there exist $x_{0},x_{1}\in
E $ and a bounded open neighborhood $\Omega $ of $x_{0}$ such that $%
x_{1}\notin \overline{\Omega }$ and 
\begin{equation*}
\max \{J(x_{0}),J(x_{1})\}<\inf_{x\in \partial \Omega }J(x).
\end{equation*}%
Let 
\begin{equation*}
\Gamma =\{h\in C([0,1],E):h(0)=x_{0},h(1)=x_{1}\}
\end{equation*}%
and 
\begin{equation*}
c=\inf_{h\in \Gamma }\max_{s\in \lbrack 0,1]}J(h(s)).
\end{equation*}%
Then $c$ is a critical value of $J$; that is, there exists $x^{\star }\in E$
such that $J^{\prime }(x^{\star })=0$ and $J(x^{\star })=c$, where $c>\max
\{J(x_{0}),J(x_{1})\}$.\bigskip
\end{lemma}

Finally we recall the Karush-Kuhn-Tucker theorem with Slater qualification
conditions (for one constraint), see \cite{borwein}:

\begin{theorem}
\label{KKT-THEO} Let $X$ be a finite-dimensional Euclidean space, $%
\eta,\mu:X\to\mathbb{R}$ be differentiable functions, with $\mu$ convex and $%
\inf_X \mu<0$, and $S=\{x\in X:\mu(x)\leq 0\}$. Moreover, let $\overline{x}%
\in S$ be such that $\eta(\overline{x})=\inf_S\eta$. Then, there exists $%
\sigma\geq 0$ such that 
\begin{equation*}
\eta ^{\prime }(\overline{x})+\sigma\mu ^{\prime }(\overline{x})=0%
\mbox{\ \
and \ \ }\sigma\mu(\overline{x})=0.
\end{equation*}
\end{theorem}

We will provide now some results which are used in the proof of the Main
Theorem. The following lemma may be viewed as a kind of a discrete maximum
principle.

\begin{lemma}
\label{lem4} Assume that $y\in Y$ is a solution of the equation%
\begin{equation}
\left\{ 
\begin{array}{l}
\Delta \left( |\Delta y(k-1)|^{p(k-1)-2}\Delta y(k-1)\right)
+f(k,y_{+}(k))=0,k\in \left[ 1,T\right] ,\bigskip \\ 
y(0)=y(T+1)=0,%
\end{array}%
\right.  \label{UKL2}
\end{equation}%
then $y\left( k\right) >0$ for all $k\in \left[ 1,T\right] $ and moreover $y$
is a solution of \eqref{zad}.
\end{lemma}

\begin{proof}
We will show that 
\begin{equation*}
\Delta y(k-1)\Delta y_{-}(k-1)\leq 0\ \ \text{\textit{for every}}\ \ k\in
\lbrack 1,T+1].\bigskip
\end{equation*}%
Indeed,%
\begin{equation*}
\begin{array}{l}
\Delta y(k-1)\Delta y_{-}(k-1)=(y(k)-y(k-1))(y_{-}(k)-y_{-}(k-1))=\bigskip
\\ 
\left[ \left( y_{+}(k)-y_{+}(k-1)\right) -\left( y_{-}(k)-y_{-}(k-1)\right) %
\right] (y_{-}(k)-y_{-}(k-1))=\bigskip \\ 
\left( y_{+}(k)-y_{+}(k-1)\right) \left( y_{-}(k)-y_{-}(k-1)\right)
-(y_{-}(k)-y_{-}(k-1))^{2}=\bigskip \\ 
y_{+}(k)y_{-}(k)-y_{+}(k)y_{-}(k-1)-y_{+}(k-1)y_{-}(k)+y_{+}(k-1)y_{-}(k-1)-%
\bigskip \\ 
(y_{-}(k)-y_{-}(k-1))^{2}=\bigskip \\ 
-\left[ y_{+}(k)y_{-}(k-1)+y_{+}(k-1)y_{-}(k)+(y_{-}(k)-y_{-}(k-1))^{2}%
\right] \leq 0.%
\end{array}%
\end{equation*}%
Assume that $y\in Y$ is a solution of (\ref{UKL2}). Taking $v=y_{-}$ in (\ref%
{functional1}) we obtain%
\begin{equation*}
\sum\limits_{k=1}^{T+1}|\Delta y(k-1)|^{p(k-1)-2}\Delta y(k-1)\Delta
y_{-}(k-1)=\sum\limits_{k=1}^{T}f(k,y_{+}(k))y_{-}(k).\bigskip
\end{equation*}%
Since the term on the left is non-positive and the one on the right is
non-negative, so this equation holds true if {the both terms are equal zero,
which} leads to $y_{-}(k)=0$ for all $k\in \lbrack 1,T].$ Then $y=y_{+}.$
Therefore $y$ is a positive solution of \eqref{zad}. Arguing by
contradiction, assume that there exists $k\in \lbrack 1,T]$ such that $%
y(k)=0 $, while we can assume $y(k-1)>0$. Then, by \eqref{UKL2} we have 
\begin{equation*}
|y(k+1)|^{p(k)-2}y(k+1)=-y(k-1)^{p(k-1)-1}-f(k,0)<0,
\end{equation*}%
which implies $y(k+1)<0$, a contradiction. So $y(k)>0$ for all $k\in \lbrack
1,T]$.
\end{proof}

\bigskip

Finally we prove that $J$ satisfies Palais-Smale condition.

\begin{lemma}
\label{lem3} Assume that \textbf{(C.1)} holds. Then the functional $J$
satisfies Palais-Smale condition.
\end{lemma}

\begin{proof}
Assume that $\{y_{n}\}$ is such that $\{J(y_{n})\}$ is bounded and $%
J^{\prime }(y_{n})\rightarrow 0$. Since $Y$ is finitely dimensional, it is
enough to show that $\{y_{n}\}$ is bounded. Note that 
\begin{equation*}
\Delta y_{+}(k)\Delta y_{-}(k)\leq 0\text{ \ \ for every \ \ }k\in \lbrack
0,T].
\end{equation*}%
Using the above inequality we obtain%
\begin{equation}
\begin{array}{l}
-\sum\limits_{k=1}^{T+1}|\Delta y(k-1)|^{p(k-1)-2}\Delta y(k-1)\Delta
y_{-}(k-1)=\bigskip \\ 
-\sum\limits_{k=1}^{T+1}|\Delta y(k-1)|^{p(k-1)-2}\Delta
(y_{+}(k-1)-y_{-}(k-1))\Delta y_{-}(k-1)=\bigskip \\ 
-\sum\limits_{k=1}^{T+1}|\Delta y(k-1)|^{p(k-1)-2}\Delta y_{+}(k-1)\Delta
y_{-}(k-1)\bigskip + \\ 
\sum\limits_{k=1}^{T+1}|\Delta y(k-1)|^{p(k-1)-2}\Delta y_{-}(k-1)\Delta
y_{-}(k-1)\geq \bigskip \\ 
\sum\limits_{k=1}^{T+1}|\Delta y(k-1)|^{p(k-1)-2}\left( \Delta
y_{-}(k-1)\right) ^{2}\geq \sum\limits_{k=1}^{T+1}|\Delta
y_{-}(k-1)|^{p(k-1)}.%
\end{array}
\label{eq1}
\end{equation}%
\newline
Since $y_{n}=\left( y_{n}\right) _{+}-\left( y_{n}\right) _{-}$we will show
that $\left\{ (y_{n})_{-}\right\} $ and $\left\{ (y_{n})_{+}\right\} $ are
bounded. Suppose that $\left\{ (y_{n})_{-}\right\} $ is unbounded. Then we
may assume that there exists $N_{0}>0$ such that for $n\geq N_{0}$ we have $%
\left\Vert (y_{n})_{-}\right\Vert \geq T\geq 2$. Using \eqref{eq1} we obtain 
\begin{equation*}
\begin{array}{l}
\left\langle J^{\prime }(y_{n}),(y_{n})_{-}\right\rangle
=\sum\limits_{k=1}^{T+1}|\Delta y_{n}(k-1)|^{p(k-1)-2}\Delta
y_{n}(k-1)\Delta (y_{n})_{-}(k-1)\bigskip \\ 
-\sum\limits_{k=1}^{T}f(k,(y_{n})_{+}(k))(y_{n})_{-}(k)\leq
-\sum\limits_{k=1}^{T+1}|\Delta (y_{n})_{-}(k-1)|^{p(k-1)}.%
\end{array}%
\end{equation*}%
So by \textbf{(A.1)}\ we obtain 
\begin{equation*}
\begin{array}{l}
T^{\frac{2-p^{-}}{2}}\Vert (y_{n})_{-}\Vert ^{p^{-}}-T\leq
\sum\limits_{k=1}^{T+1}|\Delta (y_{n})_{-}(k-1)|^{p(k-1)}\leq \bigskip \\ 
\langle J^{\prime }(y_{n}),-(y_{n})_{-}\rangle \leq \Vert J^{\prime
}(y_{n})\Vert \cdot \Vert (y_{n})_{-}\Vert .\bigskip%
\end{array}%
\end{equation*}%
Next, we see 
\begin{equation*}
\begin{array}{l}
T^{\frac{2-p^{-}}{2}}\Vert (y_{n})_{-}\Vert ^{p^{-}}\leq \Vert J^{\prime
}(y_{n})\Vert \cdot \Vert (y_{n})_{-}\Vert +T\leq \bigskip \\ 
\Vert J^{\prime }(y_{n})\Vert \cdot \Vert (y_{n})_{-}\Vert +\Vert
(y_{n})_{-}\Vert \leq \left( \Vert J^{\prime }(y_{n})\Vert +1\right) \Vert
(y_{n})_{-}\Vert \bigskip%
\end{array}%
\end{equation*}%
and%
\begin{equation*}
T^{\frac{2-p^{-}}{2}}\Vert (y_{n})_{-}\Vert ^{p^{-}-1}\leq \left( \Vert
J^{\prime }(y_{n})\Vert +1\right) .\bigskip
\end{equation*}%
Since for a fixed $\varepsilon >0$ there exists some $N_{1}\geq N_{0}$ such
that $\Vert J^{\prime }(y_{n})\Vert <\varepsilon $ for every $n\geq N_{1}$,
we get 
\begin{equation*}
\Vert (y_{n})_{-}\Vert ^{p^{-}-1}\leq \frac{\left( \varepsilon +1\right) }{%
T^{\frac{2-p^{-}}{2}}}\text{.}
\end{equation*}%
This means that $\left\{ (y_{n})_{-}\right\} $ is bounded.$\bigskip $

Now, we will show that $\{(y_{n})_{+}\}$ is bounded. Suppose that $%
\{(y_{n})_{+}\}$ is unbounded. We may assume that $\Vert (y_{n})_{+}\Vert
\rightarrow \infty $. Since 
\begin{equation*}
f(k,y)\geq \varphi _{1}(k)|y|^{m-2}y+\psi _{1}(k)\text{ \ \ \ for all}\ k\in
\lbrack 1,T],
\end{equation*}%
then 
\begin{equation*}
F(k,y)\geq \frac{\varphi _{1}(k)}{m}|y|^{m}+\psi _{1}(k)y.
\end{equation*}%
Thus by \textbf{(A.3) }and\textbf{\ (A.5)} we obtain 
\begin{equation*}
\sum_{k=1}^{T}F(k,(y_{n})_{+}(k))\geq \frac{\varphi _{1}^{-}}{m}%
\sum_{k=1}^{T}|(y_{n})_{+}(k)|^{m}\geq \frac{\varphi _{1}^{-}}{m}%
2^{-m}(T+1)^{\frac{2-m}{2}}\left\Vert (y_{n})_{+}\right\Vert ^{m},\bigskip
\end{equation*}%
where $\varphi _{1}^{-}=\min_{k\in \left[ 1,T\right] }\varphi _{1}\left(
k\right) .$ Therefore by \textbf{(A.4),}{\ }we have 
\begin{equation*}
\begin{array}{l}
J(y_{n})=\sum\limits_{k=1}^{T+1}\left[ \frac{1}{p(k-1)}|\Delta
y_{n}(k-1)|^{p(k-1)}-F(k,(y_{n})_{+}(k))\right] \leq \bigskip \\ 
2^{p^{+}}(T+1)\left( C_{p^{+}}\left\Vert \allowbreak \left( y_{n}\right)
_{+}-\left( y_{n}\right) _{-}\right\Vert ^{p^{+}}+1\right) -\frac{\varphi
_{1}^{-}}{m}2^{-m}(T+1)^{\frac{2-m}{2}}\left\Vert (y_{n})_{+}\right\Vert
^{m}\leq \bigskip \\ 
2^{p^{+}}(T+1)\left( C_{p^{+}}2^{p^{+}-1}\left( \left\Vert \allowbreak
\left( y_{n}\right) _{+}\left\Vert ^{p^{+}}+\right\Vert \left( y_{n}\right)
_{-}\right\Vert ^{p^{+}}\right) +1\right) -\bigskip \\ 
\frac{\varphi _{1}^{-}}{m}2^{-m}(T+1)^{\frac{2-m}{2}}\left\Vert
(y_{n})_{+}\right\Vert ^{m}\bigskip .%
\end{array}%
\end{equation*}%
Since $p^{+}<m$ {and $\{(y_{n})_{+}\}$ is unbounded and $\{(y_{n})_{-}\}$ is
bounded, so }$J(y_{n})\rightarrow -\infty $. Thus we obtain a contradiction
with the assumption $\{J(y_{n})\}$ is bounded, so $\{(y_{n})_{+}\}$ is
bounded. It follows that $\{y_{n}\}$ is bounded.
\end{proof}

\section{Proof of the main result}

In this section we present the proof of Theorem \ref{maintheorem}. $\bigskip 
$

\begin{proof}
Assume that $y_{0}\in Y$\ is a local minimizer of $J$\ in\textit{\ }%
\begin{equation*}
B:=\{y\in Y:\mu \left( y\right) \leq 0\},
\end{equation*}%
where $\mu (y)=\frac{\Vert y\Vert^2}{2}-\frac{1}{2(T+1)}$. Note that for $%
y\in B$ by \textbf{(A.6)} it follows that for all $k\in \lbrack 1,T]$%
\begin{equation*}
\left\vert y\left( k\right) \right\vert \leq \max_{s\in \lbrack
1,T]}\left\vert y\left( s\right) \right\vert \leq \sqrt{T+1}\left\Vert
y\right\Vert \leq \frac{1}{\sqrt{T+1}}\sqrt{T+1}=1.
\end{equation*}%
We prove that $y_0\in Int B$, by contradiction. Thus suppose otherwise, i.e.
we suppose that $y_{0}\in \partial B$. Then by Theorem \ref{KKT-THEO} there
exists $\sigma \geq 0$ such that for all $v\in Y$

\begin{equation*}
\langle J^{\prime }(y_{0}),v\rangle +\sigma \langle y_{0},v\rangle =0.
\end{equation*}%
\texttt{\ }Hence%
\begin{equation*}
\begin{array}{l}
\sum_{k=1}^{T+1}|\Delta y_{0}(k-1)|^{p(k-1)-2}\Delta y_{0}(k-1)\Delta
v(k-1)-\bigskip \\ 
\bigskip \sum_{k=1}^{T}f(k,(y_{0})_{+}(k))v(k)+\sigma \sum_{k=1}^{T}\langle
y_{0}\left( k\right) ,v\left( k\right) \rangle =0.%
\end{array}%
\end{equation*}%
\newline
Taking $v=y_{0}$, we see that 
\begin{equation*}
\sum_{k=1}^{T+1}|\Delta y_{0}(k-1)|^{p(k-1)}+\sigma \Vert y_{0}\Vert
^{2}=\sum_{k=1}^{T}f(k,(y_{0})_{+}(k))y_{0}(k).
\end{equation*}%
Since $y_{0}\in \partial B,$ we see that $\left\Vert y_{0}\right\Vert =\frac{%
1}{\sqrt{T+1}}$. Thus by \textbf{(A.2)} we have

\begin{equation*}
\sum_{k=1}^{T+1}|\Delta y_{0}(k-1)|^{p(k-1)}+\sigma \Vert y_{0}\Vert
^{2}\geq \sum\limits_{k=1}^{T+1}|\Delta y_{0}(k-1)|^{p(k-1)}\geq T^{\frac{%
p^{+}-2}{2}}\left( \frac{1}{\sqrt{T+1}}\right) ^{p^{+}}.
\end{equation*}%
On the other hand 
\begin{equation*}
\begin{array}{l}
\sum\limits_{k=1}^{T}f(k,(y_{0})_{+}(k))y_{0}(k)= \\ 
\sum\limits_{k=1}^{T}f(k,(y_{0})_{+}(k))(y_{0})_{+}(k)-\sum%
\limits_{k=1}^{T}f(k,(y_{0})_{+}(k))(y_{0})_{-}(k)\leq \\ 
\sum\limits_{k=1}^{T}\varphi _{2}(k)|\left( y_{0}\right)
_{+}(k)|^{m}+\sum\limits_{k=1}^{T}\psi _{2}(k)|\left( y_{0}\right)
_{+}(k)|\leq \sum\limits_{k=1}^{T}\varphi _{2}(k)+\sum\limits_{k=1}^{T}\psi
_{2}(k).%
\end{array}%
\end{equation*}%
Thus%
\begin{equation*}
T^{\frac{p^{+}-2}{2}}\left( \frac{1}{\sqrt{T+1}}\right) ^{p^{+}}\leq
\sum\limits_{k=1}^{T}\left( \varphi _{2}(k)+\psi _{2}(k)\right) .
\end{equation*}%
{A contradiction with the assumption }\textbf{(C.2)}. Hence $y_{0}\in IntB$
and $y_{0}$ is a local minimizer of $J$. Thus $J(y_{0})<\min_{y\in \partial
B}J(y)$.\ We will show that there exists $y_{1}$ such that $y_{1}\in
Y\setminus B$ and $J(y_{1})<\min_{y\in \partial B}J(y)$. Let $y_{\lambda
}\in Y$ be define as follows: $y_{\lambda }(k)=\lambda $ for $k=1,...,T$ and 
$y_{\lambda }(0)=y_{\lambda }(T+1)=0$. Then for $\lambda >1$ we have 
\begin{equation*}
J(y_{\lambda })\leq \frac{\lambda ^{p(0)}}{p(0)}+\frac{\lambda ^{p(T)}}{p(T)}%
-\sum_{k=1}^{T}\frac{\varphi _{1}(k)\lambda ^{m}}{m}\leq \frac{\lambda
^{p^{+}}}{p(0)}+\frac{\lambda ^{p^{+}}}{p(T)}-\frac{\varphi _{1}^{-}\lambda
^{m}}{m}T-\psi _{1}^{-}\lambda T.
\end{equation*}%
Since $m>p^{+}$, then $\lim_{\lambda \rightarrow \infty }J(y_{\lambda
})=-\infty $. Thus there exists $\lambda _{0}$ with $J(y_{\lambda
_{0}})<\min_{y\in \partial B}J(y)$. By Lemma \ref{lem2} and Lemma \ref{lem3}
we obtain a critical value\ of the functional $J$ for some $y^{\star }\in
Y\setminus \partial B$. Then $y_{0}$ and $y^{\star }$ are two different
critical points of $J$ and therefore by Lemma \ref{lem4} these are positive
solutions of problem \eqref{zad}.
\end{proof}

\begin{tabular}{l}
Marek Galewski, Szymon G\l \c{a}b, Renata Wieteska \\ 
Institute of Mathematics, \\ 
Technical University of Lodz, \\ 
Wolczanska 215, 90-924 Lodz, Poland, \\ 
marek.galewski@p.lodz.pl, szymon.glab@p.lodz.pl, renata.wieteska@p.lodz.pl%
\end{tabular}

\end{document}